\documentclass[12pt]{amsart}
\usepackage{graphicx}
\usepackage[latin1]{inputenc}
\usepackage{amsmath,amssymb,amsfonts,mathrsfs}
\usepackage{amsthm,enumerate}

\usepackage{psfrag,color}
\usepackage[small,hang,bf,up,sf]{caption}

\newcounter{lletres}

\newtheorem{Talpha}[lletres]{Theorem}
\newtheorem{Lalpha}[lletres]{Lemma}

\newtheorem{thm}{Theorem}

\newtheorem{lma}[thm]{Lemma}

\newtheorem{obs}[thm]{Remark}
\newtheorem{cor}[thm]{Corollary}

\newcommand{\D}{\mathbb D}

\newcommand{\N}{\mathbb N}
\newcommand{\C}{\mathbb C}

\newcommand{\U}{\mathbb B}
\newcommand{\Ual}{\mathcal U^{1, \alpha}}
\newcommand{\Uw}{\mathcal U^1_w}
\newcommand{\Ba}{\mathcal B_\alpha}
\newcommand{\Bh}{\mathcal B_h}
\newcommand{\dm}{d\theta}
\newcommand{\Hi}{H^\infty}
\newcommand{\Hp}{H^{\alpha}}
\newcommand{\Hw}{H^1_w}
\newcommand{\Lw}{L^1_w}
\newcommand{\Ha}{H^{1, \alpha}}
\newcommand{\Int}{\operatorname{Int}}
\newcommand*\Lap{\mathop{}\!\mathbin\bigtriangleup}

\begin{document}

\title[Extremal solutions of the Nevalinna-Pick problem and...]{Extremal solutions of Nevalinna-Pick problems and certain classes of inner functions.}

\author[Nacho Monreal Gal\'{a}n]{Nacho Monreal Gal\'{a}n}
\address{Nacho Monreal Gal\'{a}n, Departament of Mathematics, University of Crete, Voutes Campus, 70013 Heraklion, Crete, Greece.}
\email{nacho.mgalan@gmail.com}
\author[Artur Nicolau]{Artur Nicolau}
\address{Artur Nicolau, Departament de Matem\`{a}tiques, Universitat Aut\`{o}\-no\-ma de Barcelona, 08193 Bellaterra, Barcelona, Spain.}
\email{artur@mat.uab.cat}

\thanks{
Both authors are supported in part by the projects MTM2011-24606 and
2009SGR420. The first author is also supported by the research project PE1(3378) implemented within the framework of the Action ``Supporting Postdoctoral Researchers'' of the Operational Program ``Education and Lifelong Learning'' (Action's Beneficiary: General Secretariat for Research and Technology), co-financed by the European Social Fund (ESF) and the Greek State.}

\begin{abstract}
Consider a scaled Nevanlinna-Pick interpolation problem and let $\Pi$ be the Blaschke product whose zeros are the nodes of the problem. It is proved that if $\Pi$ belongs to a certain class of inner functions, then the extremal solutions of the problem or most of them, are in the same class. Three different classical classes are considered: inner functions whose derivative is in a certain Hardy space, exponential Blaschke products  and also the well known class of $\alpha$-Blaschke products, for $0<\alpha<1$.
\end{abstract}

\maketitle

\section{Introduction.}

\noindent Let $\Hi$ be the space of bounded analytic functions in
the open unit disc $\D$ of the complex plane, and let $\U$
be the  set of functions $f \in \Hi$ with $\|f\|_\infty = \sup \{|f(z)| :
z \in \D\} \leq 1$. Given two sequences $\{z_n\}$ and $\{w_n\}$ in
the unit disc $\D$, the Nevanlinna-Pick interpolation problem
consists in the following:
\begin{equation}\label{NP}
\text{Find }f\in\U\text{ such that }f(z_n)=w_n\text{ for }n\in\N.
\end{equation}
Nevanlinna \cite{Nev19} and Pick \cite{Pic15} independently
considered this problem. It was proved that there exists a
solution if and only if for any integer $N \geq 1$, the matrix
$$\left(\frac{1-w_i\overline w_j}{1-z_i\overline z_j}\right)_{i,j=1,\dots,N}$$
is positive semidefinite.  When the problem (\ref{NP}) has more than one solution,
Nevalinna showed that all the solutions can be
expressed in the following way:

\begin{equation}\label{parametrization}
 \left\{f\in\U\ :\ f\text{ solves }(\ref{NP})\right\}=\left\{\frac{P-Q\varphi}{R-S\varphi} : \varphi\in\U\right\}.
\end{equation}

\noindent This parametrization arose from an iterative argument
called Schur's algorithm, which is explained in detail for example
in \cite[p. 159]{Gar07}. The Nevanlinna coefficients $P,\ Q,\ R$ and $S$ are
analytic functions in  the Smirnov class, depend on the sequences
$\{z_n\}$ and $\{w_n\}$ and are uniquely determined by these
sequences if one normalizes them so that $S(0)=0$ and $PS- QR=\Pi$.
Here $\Pi$ denotes the Blaschke product with zeros $\{z_n\}$, that
is,
$$\Pi(z)=\prod_n\frac{|z_n|}{z_n}\frac{z_n-z}{1-\overline z_n z}.$$
Observe that the Blaschke condition $\sum (1-|z_n|) < \infty $ holds because (\ref{NP}) has more than one solution.  Among all the solutions of (\ref{NP}), the ones of the form
$$I_\gamma(z)=\frac{P(z)-Q(z)e^{i\gamma}}{R(z)-S(z)e^{i\gamma}}\text{, with }\gamma\in[0,2\pi),$$
which are called extremal solutions, are of special interest. In
\cite{Nev29} Nevanlinna proved that every extremal solution is an
inner function. Recall that a function $I\in\Hi$ is inner if the
radial limits $I(e^{i\gamma})= \lim_{r \to 1} I(re^{i\gamma})$
satisfy $|I(e^{i\gamma})|=1$ for a.e. $\gamma\in[0,2\pi)$. This was
refined by Stray in \cite{Str88}, who proved that $I_\gamma$ is a
Blaschke product a.e.(C) $\gamma\in[0,2\pi)$, that is, for all
$\gamma\in[0,2\pi)$ except for a set of logarithmic capacity zero.
This result is related to a classical theorem by Frotsman (see
\cite[p. 75]{Gar07}), which says that if $I$ is an inner function,
then a.e.(C) $w\in\D$ the function $(I-w)/(1-\overline w I)$ is a
Blaschke product.
\bigskip

\noindent Consider a Nevanlinna-Pick interpolation problem (\ref{NP}) with more
than one solution, and its corresponding Blaschke product $\Pi$. It
is natural to expect that if a certain property is held by $\Pi$
then it is also enjoyed by the corresponding extremal solutions of
(\ref{NP}), or by most of them. This question will be considered for certain natural subclasses of inner functions that will be introduced below.

\noindent For $0< \alpha < 1$, let
$\Ba$ denote the class of Blaschke products $B$ for which
$$\sum_{z\,:\,B(z)=0}(1-|z|)^{1-\alpha}<\infty.$$
Carleson considered this class of Blaschke products in his doctoral
thesis \cite{Car50}. Among many other results he proved the
following one (see  \cite[p. 28]{Car50}).

\begin{Lalpha}\label{carleson}
Let $B$ be a Blaschke product. Fix $0< \alpha <1$. Then $B\in\Ba$ if and only if
$$\int_\D\frac{\log|B(z)|^{-1}}{(1-|z|)^{1+\alpha}}dA(z)<\infty.$$
\end{Lalpha}

\noindent Here $dA(z)$ denotes the area measure in $\D$. In the
spirit of Frostman's result, Carleson proved also in \cite[p.
56]{Car50} that if $B\in\Ba$, their Frostman shifts
$(B-w)/(1-\overline w B)$ are also in $\Ba$ a.e.(C) $w\in\D$.

\medskip

\noindent Let $\Ha$ denote the class of analytic functions in the
unit disc whose derivative is in the Hardy space $\Hp$. Recall that an
analytic function $f$ is in $\Hp$ if
$$\|f\|^\alpha_{\alpha}=\sup_{0<r<1}\int_0^{2\pi}|f(re^{i\theta})|^\alpha \dm<+\infty.$$

\noindent When $\alpha \geq 1$, the only inner functions in $\Ha$
are the finite Blaschke products. When $\alpha < 1$, inner functions
$I$ in $\Ha$ have been extensively studied (see \cite{AheCla74}, \cite{Dya94}, \cite{Dya98}, \cite{Dya06}, \cite{FriMas08} and \cite{Pro73}). In \cite{Pro73} (see also \cite{AheCla74}), it
was shown that $\Ba\subset\Ha$ for $\alpha\in(1/2,1)$. The converse
is false, although, fixed $1/2 < \alpha < 1$, Ahern proved in
\cite{Ahe79} that an inner function $I$ is in $\Ha$  if and only if
there exists $w\in\D$ such that the Frostman shift
$(I-w)/(1-\overline w I)$ is in $\Ba$. There is another
characterization of $\Ha$ for $\alpha\in(1/2,1)$ in terms of
Carleson contours given by Cohn in \cite{Coh83}. Carleson contours
are described in the following result in \cite{Car62} or
\cite[p.333]{Gar07}

\begin{Talpha}\label{contour}
 Let $f\in\U$. Given $0<\varepsilon<1$ there exists $0< \delta(\varepsilon) < 1$ and a
 collection $\{\Gamma_j\}$ of closed curves in $\D$ such that
 \begin{enumerate}[\bf (i)]
  \item $|f(z)|>\varepsilon$ if
  $z\in\D\setminus\operatornamewithlimits{\bigcup}\limits_j\Int\Gamma_j$,
  \item $|f(z)|<\delta$ if
  $z\in\operatornamewithlimits{\bigcup}\limits_j\Int\Gamma_j$,
  \item Arc length in
  $\operatornamewithlimits{\bigcup}\limits_j \Gamma_j$ is a
  Carleson measure.
 \end{enumerate}
\end{Talpha}
\noindent We will refer to
$\Gamma=\operatornamewithlimits{\bigcup}\limits_j \Gamma_j$ as
$\varepsilon$-Carleson contour of the function $f$ .
\medskip

\noindent Let us consider another class of Blaschke products. A Blaschke product $B$ is said to be
exponential if there exists a constant $M=M(B)>0$ such that for any
$j=1,2,\ldots$, the annulus $\{z\in\D\,:\,2^{-j-1}<1-|z|<2^{-j}\}$
does not contain more than $M$ zeros of $B$. Clearly, an exponential
Blaschke product is in $\Ba$ for any $0<\alpha<1$. Recall that an
analytic function in the unit disc $f$ is in the weak Hardy space $\Hw$ if its
non-tangential maximal function is in $\Lw(\partial\D)$, that is, if
there exists a constant $C=C(f)>0$ such that for every $\lambda>0$
one has that
$$\left|\left\{e^{i\gamma}\in\partial\D\,:\,Mf(e^{i\gamma})>\lambda)\right\}\right|\leq C/\lambda.$$
\noindent Here $Mf$ denotes the non-tangential maximal function of $f$ and $|E|$ denotes the length of the set
$E\subset \partial\D$. In \cite{CimNic12} the following
characterization of exponential Blaschke products was given.

\begin{Lalpha}\label{cimnic}
 An inner function $I$ is an exponential Blaschke product if and only if $I'$ is in the weak Hardy space $\Hw$.
\end{Lalpha}

\noindent A similar result holds in the context of weak Besov
spaces. See \cite{GroNic13}.
\bigskip

\noindent Returning to the interpolation problem (\ref{NP}), one may
ask how the extremal solutions are when $\Pi$ belongs to one of
these classes of functions. The methods used here depend
strongly of considering a subclass of Nevanlinna-Pick problems,
called scaled problems. A
Nevanlinna-Pick problem is called scaled if there exists a solution $f_0$
of the problem with $\|f_0\|_\infty<1$.

\begin{thm}\label{main}
Let (\ref{NP}) be a scaled Nevanlinna-Pick problem and let $\Pi$ be the Blaschke product with zeros $\{z_n\}$. Let $I_\gamma$, with $\gamma\in[0,2\pi)$, be the extremal solutions of (\ref{NP}).
\begin{enumerate}[\bf (a)]
 \item Fix $0 < \alpha <1$. Assume that $\Pi\in\Ha$, then $I_\gamma\in\Ha$ for every $\gamma\in[0,2\pi)$.
 \item Assume that $\Pi$ is an exponential Blaschke product, then $I_\gamma$ is an exponential Blaschke product for every $\gamma\in[0,2\pi)$.
 \item Assume that
 \begin{equation}\label{extra}
  \sum_{n}(1-|z_n|)^{1-\alpha}|\log(1-|z_n|)|<\infty,
 \end{equation}
 then $I_\gamma\in\Ba$ a.e.(C) $\gamma\in[0,2\pi)$.
\end{enumerate}
\end{thm}

\noindent The conclusion of part \textbf{(c)} may still hold under the weaker assumption $\Pi \in \Ba$, but we do not know how to prove it. It is also worth mentioning that in contrast with the statements in \textbf{(a)} and \textbf{(b)}, in part \textbf{(c)} an exceptional set is needed. Actually, pick a Blaschke product $\Pi \in \Ba $ with zeros $\{z_n\}$ such that for a certain $w \in \D$ one has $(\Pi + w) /(1 + \overline{w} \Pi) \notin \Ba$. Consider the Nevanlinna-Pick problem (\ref{NP}) with $w_n = w$ for any $n$. For this specific problem one has that
\begin{equation*}
 \left\{f\in\U\ :\ f\text{ solves }(\ref{NP})\right\}=\left\{\frac{w+\Pi \varphi}{1 + \overline{w} \Pi \varphi} : \varphi\in\U\right\}.
\end{equation*}
Hence not all the extremal solutions $I_{\gamma} = (w+ \Pi e^{i\gamma} ) / (1+  \overline{w} e^{i\gamma} \Pi )$ can be in $\Ba$ and an exceptional set is needed in the statement \textbf{(c)}. A direct consequence of the Theorem 1 is the following:

\begin{cor}
 Let $\Pi$ be a Blaschke product with zeros $\{z_n\}$ and let $\{w_n\}$ be a sequence of complex numbers such that
 $$M=\inf\left\{\|f\|_\infty\,:\,f\in\Hi,\,f(z_n)=w_n,\, n=1,2,\dots\right\} < \infty .$$
 Fix $\varepsilon>0$.
 \begin{enumerate}[\bf (a)]
  \item If $\Pi\in\Ha$ then there exists a Blaschke product $I\in\Ha$ such that $(M+\varepsilon)I(z_n)=w_n$ for $n=1,2,\dots$
  \item If $\Pi$ is an exponential Blaschke product then there exists an exponential Blaschke product $I$ such that $(M+\varepsilon)I(z_n)=w_n$ for $n=1,2,\dots$
  \item If $\operatornamewithlimits{\sum}\limits_{n}(1-|z_n|)^{1-\alpha}|\log(1-|z_n|)|<\infty$ then there exists $I\in\Ba$ such that $(M+\varepsilon)I(z_n)=w_n$ for $n=1,2,\dots$
 \end{enumerate}
\end{cor}

\noindent Section 2 will be devoted to the proof of Theorem \ref{main}.

\bigskip

\noindent In the same spirit of Theorem 1, it will be proved in Section 3 that certain ratios
of Nevanlinna's coefficients of a scaled problem belong to $\Ha$
(respectively, have derivatives in $H^1_w$) whenever $\Pi \in \Ha$
(respectively, $\Pi ' \in \Hw$). To this end, let $\Ual$ (respectively, $\Uw$) denote the closed (in $\Hi$) linear hull of the inner functions in $\Ha$ (respectively, inner functions whose derivative is in $\Hw$). The next theorem is analogue to the main result in
\cite{NicStr96}.

\begin{thm}\label{thm2}
 Let (\ref{NP}) be a scaled Nevanlinna-Pick problem and consider Nevanlinna's parametrization (\ref{parametrization}). Let $\Pi$ be the Blaschke product with zeros $\{z_n\}$.
 \begin{enumerate}[\bf (a)]
  \item Fix $0<\alpha <1$. Assume that $\Pi\in\Ha$. Then  $Q/R$ is an inner function in $\Ha$ and the functions $P/R$,  $S/R$ and $1/R$ are in $\Ual$.
  \item Assume that $\Pi$ is an exponential Blaschke product. Then $Q/R$ is an exponential Blaschke product and the functions $P/R$, $S/R$ and $1/R$ are $\Uw$.
 \end{enumerate}
\end{thm}

\noindent From now on, the letter $C$ will denote a universal
constant, while $C(x)$ will denote a constant depending on the
parameter $x$.

\bigskip

\noindent It is a pleasure to thank Mihalis Papadimitrakis for helpful discussions.

\bigskip

\section{Proof of Theorem \ref{main}.}
\noindent Given two sequences $\{z_n\}$ and $\{w_n\}$ in the unit
disc, consider the corresponding Nevanlinna-Pick problem (\ref{NP})
and suppose that it has more than one solution. Then, if $f$ denotes
one of these solutions, according to Nevanlinna's parametrization,
$f$ may be expressed as
\begin{equation}
 f(z)=\frac{P(z)-Q(z)\varphi(z)}{R(z)-S(z)\varphi(z)}
\end{equation}
\noindent for a certain $\varphi\in\U$. Fix a point  $z\in\D$ and consider the set
$\Delta(z)=\{f(z)\,:\,f\text{ solves }(\ref{NP})\}$. By Nevanlinna's parametrization one has that
$$\Delta(z)=\left\{\frac{P(z)-Q(z)w}{R(z)-S(z)w},\ w\in\overline\D\right\}.$$
Then $\Delta(z)$ is a disc contained in $\D$, which is called
vertevorrat, with center $c(z)$ and radius $\rho(z)$ given by
\begin{eqnarray}\label{center}
 \hspace{1cm} c(z)=\frac{P(z)\overline {R(z)}-Q(z)\overline {S(z)}}{|R(z)|^2-|S(z)|^2}, \hspace{1cm} \rho(z)=\frac{|\Pi(z)|}{|R(z)|^2-|S(z)|^2}
\end{eqnarray}
\noindent Observe that $I_\gamma(z)\in\partial\Delta(z)$ for all
$\gamma\in[0,2\pi)$. Nevanlinna also proved that $|R(z)|>
\max \{|Q(z)|, |P(z)|, |S(z)|, 1 \}$ for any $z \in \D$ and that at
almost every point of the unit circle one has $Q=- \Pi \overline{R}
$ and $P=- \Pi \overline{S} $. In \cite{Str91} Stray proved that if
the problem is scaled the vertevorrat, and also
Nevanlinna's coefficients, have certain nice properties.

\begin{Lalpha}\label{scaled}
 Assume that (\ref{NP}) is scaled. Then
 \begin{enumerate}[\bf (i)]
  \item There exists $\eta>0$ such that $(1-\eta)|\Pi(z)| \leq \rho(z) \leq |\Pi(z)| $ for any $z \in \D$.
  \item The radius $\rho(z)\rightarrow1$ as $|\Pi(z)|\rightarrow1$.
  \item There exists $\epsilon>0$  such that $P$, $Q$, $R$ and $S$ are in $H^{2+\epsilon}$.
 \end{enumerate}
\end{Lalpha}
\medskip

\noindent Let $f \in \U$ and $e^{i\theta} \in \partial \D$ such that
$\lim_{r \to 1} f(r e^{i \theta}) \in \partial \D$. If $f'(z)$ has
limit as $z$ approaches non-tangentially $e^{i\theta}$, the limit is
called the angular derivative of $f$ at $e^{i\theta}$ and it is
denoted by $f'(e^{i\theta})$. It is a classical fact that if
$$
\liminf_{z \to e^{i\theta}} \frac{1- |f(z)|}{1-|z|} < \infty \, ,
$$
then $f$ has an angular derivative at $e^{i\theta}$ and
$$
|f'(e^{i\theta})| = \lim_{r \to 1} \frac{1-|f(re^{i\theta})|}{1-r}
$$
 The following Lemma is a basic tool in the
proof of Theorem \ref{main}.

\begin{lma}\label{lemma2}
 Let $B$ and $I$ be two inner functions, and let $\Gamma$ be an $\varepsilon$-Carleson contour of $B$. Suppose
 that $\operatornamewithlimits{\inf}\limits_{z\in\D\setminus\Int\Gamma}|I(z)|\geq\eta>0$. Then there exists a constant $C= C(\varepsilon, \eta)>0$ such
 that for every $e^{i\theta}\in\partial\D$ for which $\lim_{r \to 1} B(re^{i\theta} ) \in \partial \D$  and $B'(e^{i\theta})$ exists, one has that $I$ has angular derivative at $e^{i\theta}$ and
 $|I'(e^{i\theta})|\leq C|B'(e^{i\theta})|$.
\end{lma}

\begin{proof}
Consider the domain $\Omega = \D \setminus \Int\Gamma $, where $\Int\Gamma = \cup \Int\Gamma_j$. The functions $\log|I|^{-1}$ and $\log|B|^{-1}$ are harmonic on $\Omega$     and they both vanish on $\partial\D$. By hypothesis there exists a constant $C=C(\varepsilon,\eta)>0$ such that $\log|I|^{-1}\leq C\log|B|^{-1}$ on $\Gamma$. Applying the maximum principle to these functions in the domain $\Omega $ one gets that
 $$\log\frac{1}{|I(z)|}\leq C\log\frac{1}{|B(z)|}$$
 when $z\in \Omega$. Then one deduces that there exists a constant $C_1>0$ such that $1-|I(re^{i\theta})|\leq C_1(1-|B(re^{i\theta})|)$ for all $re^{i\theta}\in\D\setminus\Int\Gamma$. Hence, dividing by $1-r$ in both sides and taking $r\rightarrow1$ the Lemma is proved.
\end{proof}

\medskip

\noindent Now one is ready to prove \textbf{(a)} and \textbf{(b)} in Theorem \ref{main}. Since the Nevanlinna-Pick problem (\ref{NP}) is scaled, one may apply Lemma \ref{scaled}  to see that there exists $\eta>0$ such that $\rho(z)\geq3/4$ when $|\Pi(z)|\geq1-\eta$.  Fixed $\gamma \in [0, 2 \pi)$, let $I_\gamma(z)$ be an extremal solution of (\ref{NP}), then $|I_\gamma(z)|\geq1/4$ if $|\Pi(z)|\geq1-\eta$.
Considering a $(1-\eta)$-Carleson contour of $\Pi$ one may apply Lemma \ref{lemma2} to the functions $I_\gamma$ and $\Pi$ to conclude that $I_\gamma$ has angular derivative at  almost every point $e^{i\theta} \in \partial \D$ and
\begin{equation}\label{prop1}
 |I'_\gamma(e^{i\theta})|\leq C|\Pi'(e^{i\theta})|.
\end{equation}
\noindent Now one may prove the first two paragraphs of Theorem
\ref{main}. In order to prove \textbf{(a)}, suppose that
$\Pi\in\Ha$. Then clearly (\ref{prop1}) implies that
$I_\gamma\in\Ha$ for any $\gamma\in[0,2\pi)$. To prove \textbf{(b)},
note that for every $\lambda>0$ one has that
 \begin{equation}\label{inclusion}
  \{e^{i\theta}\in\partial\D\,:\,MI'_\gamma(e^{i\theta})>\lambda\}\subseteq\{e^{i\theta}\in\partial\D\,:\,M\Pi' (e^{i\theta})>\lambda/C\}.
 \end{equation}
Since $\Pi$ is an exponential Blaschke product, one may apply Lemma
\ref{cimnic} to deduce that $M\Pi' \in \Lw (\partial \D)$ and
(\ref{inclusion}) and Lemma  \ref{cimnic} again, to conclude that $I_\gamma$ is an exponential
Blaschke product for any $\gamma\in[0,2\pi)$.

\medskip

\noindent In order to prove \textbf{(c)} of Theorem 1 one needs an
auxiliary result, which is analogue to the description of $\Ba$
given by Carleson stated as Lemma \ref{carleson}.

\begin{lma}\label{int_cond}
 Let $\{z_n\}$ be a Blaschke sequence in $\D$ and let $\Pi$ be the Blaschke product with zeros $\{z_n\}$.
  Fix $0 < \alpha <1$. The following conditions are equivalent:
 \begin{enumerate}[\bf (i)]
  \item $\operatornamewithlimits{\sum}\limits_n(1-|z_n|)^{1-\alpha}\log(1-|z_n|)^{-1}<\infty$.
  \item $\operatornamewithlimits{\int}\limits_\D\dfrac{\log|\Pi(z)|^{-1}}{(1-|z|)^{1+\alpha}}\log(1-|z|)^{-1}dA(z)<\infty$.
 \end{enumerate}
\end{lma}

\begin{proof}
 Let $\Pi_N$ be the finite Blaschke product with zeros $\{z_n : n=1,\dots,N \}$. Then by Green's formula
 \begin{multline*}
  \int_\D(1-|z|^2)^{1-\alpha}\log(1-|z|^2)\Lap\log|\Pi_N(z)|dA(z)=\\=\int_\D\Lap\left((1-|z|^2)^{1-\alpha}\log(1-|z|^2)\right)\log|\Pi_N(z)|dA(z).
 \end{multline*}
 \noindent Denoting by $\delta_{z_n}$ the Dirac measure at the point $z_n$, one may check that $\Lap\log|\Pi_N(z)|=C\operatornamewithlimits{\sum}\limits_{n=1}^N\delta_{z_n}$, in the sense of distributions. Besides
 \begin{multline*}
  \Lap\left((1-|z|^2)^{1-\alpha}\log(1-|z|^2)\right)=C(\alpha)(1-|z|^2)^{-1-\alpha}\log(1-|z|^2)+\\+O\left((1-|z|^2)^{{-1-\alpha}} \right).
 \end{multline*}
 \noindent Hence,
 \begin{multline*}
  C \sum_{n=1}^N  (1-|z_n|^2)^{1-\alpha}\log(1-|z_n|^2) =\\=C(\alpha)\int_\D\dfrac{\log|\Pi_N(z)|}{(1-|z|^2)^{1+\alpha}}\log(1-|z|^2)dA(z)+\\+O(1)\int_\D\dfrac{\log|\Pi_N(z)|}{(1-|z|^2)^{1+\alpha}}dA(z).
 \end{multline*}

\noindent The proof finishes by letting $N \to \infty$.

\end{proof}

\medskip

\noindent Return now to the scaled Nevalinna-Pick interpolation
problem (\ref{NP}) in order to prove statement \textbf{(c)} of Theorem 1.  Let
$I_\gamma$ be the corresponding extremal solutions of the problem. We follow an idea of Stray in \cite{Str88}.
Consider the set
$E=\{e^{i\gamma}\in\partial\D\,:\,I_\gamma\notin\Ba\}$. One wants to
see that $E$ has logarithmic capacity zero. So, given a positive
measure $\mu$ supported on $E$ with bounded logarithmic potential,
that is,
$$K=\sup_{z\in\C}\int_E\log|z-w|^{-1}d\mu(w)<\infty,$$
one needs to show that $\mu\equiv0$. To this end, Lemma \ref{carleson} tells that it is enough to see that
$$\int_E\int_\D\frac{\log|I_\gamma(z)|^{-1}}{(1-|z|)^{1+\alpha}}dA(z)d\mu(\gamma)<\infty.$$
By Fubini this is equivalent to show that
\begin{equation}\label{integral}
 \int_\D\left(\int_E\log|I_\gamma(z)|^{-1}d\mu(\gamma)\right)\frac{dA(z)}{(1-|z|)^{1+\alpha}}<\infty.
\end{equation}
\noindent As in paragraphs \textbf{(a)} and \textbf{(b)}, by (ii) of
Lemma D one can pick $\eta
>0$ small enough such that the $(1-\eta)$-Carleson contour $\Gamma$ of
$\Pi$ satisfies that $\rho(z) > 3/4$ for any $z \in
\D\setminus\Int\Gamma $. Hence for every $\gamma\in[0,2\pi)$ one has
that $|I_\gamma(z)|\geq1/4$ for $ z \in \D\setminus\Int\Gamma$.
Split the integral in (\ref{integral}) over the disc into two parts,
one over the set $\Int\Gamma$ and the other one over
$\D\setminus\Int\Gamma$. Since for any $\gamma \in [0 , 2\pi]$, the maximum principle gives that  $\log|I_\gamma(z)|^{-1}\leq C(\eta)\log|\Pi(z)|^{-1}$ for $ z
\in \D\setminus\Int\Gamma$, one deduces
\begin{multline*}
 \int_{\D\setminus\Int\Gamma}\left(\int_E\log|I_\gamma(z)|^{-1}d\mu(\gamma)\right)\frac{dA(z)}{(1-|z|)^{1+\alpha}}\leq \\ \leq C(\eta) \mu(E)\int_{\D\setminus\Int\Gamma}\frac{\log|\Pi(z)|^{-1}}{(1-|z|)^{1+\alpha}}dA(z)<\infty
\end{multline*}
\noindent by Lemma \ref{carleson}. Focused now on the other part of the integral, the fact that the logarithmic potential
of $\mu$ is bounded yields
$$\int_E\log |I_\gamma(z)|^{-1}d\mu(\gamma)\leq\log\frac{1}{\max\left\{|(P/R)(z)|,|(Q/R)(z)| \right\}}+2K + \log 2.$$
Since $S/R\in\U$ and $(Q/R)(S/R)-P/R=\Pi/R^2$ one has that
$$\max\left\{|(P/R)(z)|,|(Q/R)(z)| \right\} \geq |\Pi (z)|/
2|R(z)|^2$$
Hence
\begin{multline*}
 \int_{\Int\Gamma}\left(\int_E\log|I_\gamma(z)|^{-1}d\mu(\gamma)\right)\frac{dA(z)}{(1-|z|)^{1+\alpha}}\leq \\ \leq\int_{\Int\Gamma}\frac{\log|\Pi(z)|^{-1}}{(1-|z|)^{1+\alpha}}dA(z) +\int_{\Int\Gamma}\frac{\log|R(z)|^2}{(1-|z|)^{1+\alpha}}dA(z) + \\ +\int_{\Int\Gamma}\frac{\log4+2K}{(1-|z|)^{1+\alpha}}dA(z)=(A_1)+(A_2)+(A_3).
\end{multline*}
\noindent The first integral $(A_1)$ is finite again by Lemma \ref{carleson}. Lemma \ref{scaled} implies that the function $R\in H^2$, and then $\log|R(z)|\leq C(1+\log(1-|z|)^{-1})$. Hence
$$(A_2)+(A_3)\leq C\int_{\Int\Gamma}\frac{|\log(1-|z|)|}{(1-|z|)^{1+\alpha}}dA(z)+\int_{\Int\Gamma}\frac{C(K)}{(1-|z|)^{1+\alpha}}dA(z).$$
Since $\Gamma$ is a $(1-\eta)$-Carleson contour of $\Pi$, by Lemma
\ref{contour} there exists $0< \delta < 1$ depending on $\eta$ such
that $|\Pi(z)|\leq\delta$ on $\Int\Gamma$. Then,
$$(A_2)+(A_3) \leq C(\delta)\int_{\Int\Gamma}\frac{\log|\Pi(z)|^{-1}}{(1-|z|)^{1+\alpha}}|\log(1-|z|)|dA(z),$$
which is finite according to Lemma \ref{int_cond}. Then the integral
(\ref{integral}) is finite, and this finishes the proof of Theorem 1.$\hfill\qed$

\bigskip

\begin{obs}
 In the last part of the proof it is shown actually that if the sequence $\{z_n\}$ satisfies condition (\ref{extra}) then
 $$\int_{\Int\Gamma}\frac{|\log(1-|z|)|}{(1-|z|)^{1+\alpha}}dA(z)<\infty,$$
 where $\Gamma$ is a Carleson contour of $\Pi$. This is a direct consequence of Lemma \ref{int_cond}, and it
 is an analogue to a condition introduced by Cohn in \cite{Coh83} for describing inner functions in  $\Ha$ for $\alpha\in(1/2,1)$.
\end{obs}

\medskip

\begin{obs}
 One may prove a more general result than \textbf{(c)} of Theorem 1.  Let $h:[0,1]\longrightarrow[0,1]$  be
  a twice differentiable non increasing  concave function such that $h(1)=0$. We say that a Blaschke product $B$ is in the class  $\Bh$ if
 $$
 \sum_{z: B(z)=0} h(|z|^2) < \infty
 $$
 When $h(t)=(1-t)^{1- \alpha}$, $0< \alpha < 1$ the class $\Bh$ is $\Ba$. Assume that
  $\operatornamewithlimits{\limsup}\limits_{r\rightarrow1}|h'(r)/h''(r)|<1$.
 The proof of Lemma 5 shows that a Blaschke product $B$ is in $\Bh$ if and only if
 $$
 \int_\D  h''(|z|^2) \log|B(z)|^{-1}dA(z)<\infty.
 $$
 Similarly, the zeros of a Blaschke product $B$ satisfy
 $$
 \sum_{z: B(z)=0} h(|z|^2) | \log (1-|z|^2)| < \infty
 $$
if and only if
$$
 \int_\D\log|B(z)|^{-1}h''(|z|^2)|\log(1-|z|^2)|dA(z)<\infty.
$$
Now, following the proof of part \textbf{(c)} of Theorem 1, one obtains the following result.
 \begin{thm}
  Let $h:[0,1]\longrightarrow[0,1]$  be
  a twice differentiable non increasing concave function such that $h(1)=0$.
  Assume that
$$\operatornamewithlimits{\limsup}\limits_{r\rightarrow1}|h'(r)/h''(r)|<1 .$$
Let (\ref{NP}) be a scaled Nevanlinna-Pick problem and let $\Pi$ be
the Blaschke product with zeros $\{z_n\}$. Let $I_\gamma$, with
$\gamma\in[0,2\pi)$, be the extremal solutions of (\ref{NP}). Assume
that
\begin{equation}\label{log_condition}
  \sum_n h(|z_n|^2)|\log(1-|z_n|^2)|<\infty
 \end{equation}
 then $I_\gamma\in\Bh$ a.e.(C) $\gamma\in[0,2\pi)$.
\end{thm}

\end{obs}

\bigskip

\section{Proof of Theorem \ref{thm2}.}

\noindent Consider now $H^\infty$ as a subalgebra of the space $L^\infty (\partial
\D)$. Let $D_\Pi$ be the Douglas algebra generated by $H^\infty$ and
the restriction of $\overline{\Pi}$ to $\partial \D$. Let $I$ be an
inner function. The property of $I$ being invertible in $D_\Pi$
means that $|I(\xi_n)| \to 1 $ whenever $\{\xi_n \}$ is a sequence
satisfying $|\Pi (\xi_n)| \to 1$. Let $CDA_\Pi$ be the subalgebra
of $H^\infty$ generated by all inner functions $I$ invertible in
$D_\Pi$. Consider again a scaled Nevanlinna-Pick problem (\ref{NP}) and let $\Pi$ be the Blaschke product
with zeros $\{z_n\}$. In \cite{Tol93} it was shown that all the
extremal solutions of (\ref{NP}) are invertible in $D_\Pi$.
Furthermore, recalling Nevalinna's parametrization in
(\ref{parametrization}), in \cite{NicStr96} it was proved that the functions $P/R$,
$Q/R$, $S/R$ and $1/R$ belong also to $CDA_\Pi$.

\medskip

\noindent For proving paragraph \textbf{(a)} of Theorem \ref{thm2}, suppose that $\Pi\in\Ha$. It is well known that $Q/R$ is an inner function. Observe that
 \begin{multline*}
  \frac{(Q/R)(z)-(P/R)(z)\cdot(S/R)(z)}{1-\left|(S/R)(z)\right|^2}-(Q/R)(z)= \\ =(S/R)(z)\cdot\frac{(Q/R)(z)\cdot\overline{(S/R)(z)}-(P/R)(z)}{1-\left|(S/R)(z)\right|^2}
 \end{multline*}
\noindent Since the problem is scaled, Lemma \ref{scaled} gives that $c(z)\rightarrow0$ and $\rho(z)\rightarrow1$ when $|\Pi(z)|\rightarrow1$.  Hence if $\Gamma$ is an $(1- \varepsilon)$-Carleson contour of $\Pi$ for sufficiently small $0< \varepsilon < 1$, then this shows that there exists an $0< \eta < 1$ such that $|Q(z) / R (z)|>1-\eta$ when $z\in\D\setminus\Int\Gamma$. One may apply here Lemma \ref{lemma2} to conclude that $Q/R\in\Ha$.

\medskip

\noindent In order to see that the function $P/R$ is in $\Ual$, fix
$z\in\D$ and consider the following integral
$$\frac{1}{2\pi}\int_0^{2\pi}I_\gamma(z)d\gamma=\frac{1}{2\pi}\int_0^{2\pi}\frac{P(z)-Q(z)e^{i\gamma}}{R(z)-S(z)e^{i\gamma}}d\gamma,$$
where $I_\gamma$ are the extremal solutions of the problem. Consider
$I_w(z)=(P(z)-Q(z)w)/(R(z)-S(z)w)$. Since $|(S/R)(z)|<1$, one has
that $I_w(z)$ is analytic in $w\in\overline\D$. Then
$$(P/R)(z)=\frac{1}{2\pi}\int_0^{2\pi}\frac{(P/R)(z)-(Q/R)(z)e^{i\gamma}}{1-(S/R)(z)e^{i\gamma}} d \gamma.$$

\noindent Theorem \ref{main} tells that there exists a constant $C>0$ such that $\|I'_\gamma\|_\alpha\leq C\|\Pi'\|_\alpha$. Hence $(P/R) \in \Ual$.

\medskip

\noindent Let  $0<\delta<1$ be  a constant to be fixed later. The function $\tilde I_w(z)=(\delta(S/R)(z)+(Q/R)(z)w)/(1+\delta(P/R)(z)w)$ is analytic in $w\in\overline\D$. Consequently
$$\delta(S/R)(z)=\frac{1}{2\pi}\int_0^{2\pi}\frac{\delta(S/R)(z)+(Q/R)(z)e^{i\gamma}}{1+ \delta(P/R)(z)e^{i\gamma}}d\gamma.$$
Observe that $\tilde I_w(z)$ is also analytic and uniformly bounded in $z\in\D$. In order to see that $\tilde I_w$ are inner functions in $z$ for any $w \in \D$, recall that $P\overline R-Q\overline S=0$ a.e. on $\partial\D$. Then for $|w|=1$ one has that
$$|\tilde I_w(e^{i\theta})|=\left|\frac{\delta (S/R)(e^{i\theta})+ (Q/R)(e^{i\theta})w}{1+\delta\overline{(S/R)}(e^{i\theta})Q/R(e^{i\theta})w}\right|=1\,\,\text{a.e. }\theta\in[0,2\pi).$$
Take $\varepsilon>0$ and let $\Gamma$ be an $\varepsilon$-Carleson contour of $(Q/R)$. Choosing $0< \delta<\varepsilon$ one has that $|\tilde I_w(z)|\geq (\varepsilon-\delta)/2$ when $z\in\D\setminus\Int\Gamma$. Then Lemma \ref{lemma2} yields that $\tilde I_w\in\Ha$ for all $w\in\partial\D$. As before, this shows that $(S/R)\in\Ual$.

\medskip

\noindent Consider now the function $(\delta(1/R)(z)+(Q/R)(z)w)/(1-\delta(\Pi/R)(z)w)$, for a certain $0<\delta<1$ to be fixed, which is analytic in $w\in\overline\D$. A similar argument shows that for any $w \in \partial \D$, this function is inner in the variable $z$. This holds because $-(\Pi/R)=(Q/R)(1/\overline R)$ a.e. on $\partial\D$. Since
$$\delta(1/R)(z)=\frac{1}{2\pi}\int_0^{2\pi}\frac{\delta(1/R)(z)+(Q/R)(z)e^{i\gamma}}{1-\delta(\Pi/R)(z)e^{i\gamma}}d\gamma,$$
one may conclude that $(1/R)\in\Ual$.

\medskip

\noindent The proof of paragraph \textbf{(b)}  is an easy adaptation of the arguments above, so it will not be included here.$\hfill\qed$

\bigskip

\bibliographystyle{plain}
\bibliography{classes_of_Bp}

\end{document}